\documentclass{tran-l}

\usepackage{hyperref}
\usepackage{graphicx}
\usepackage{amssymb}
\usepackage{amsthm}
\usepackage{amsmath}
\usepackage{mathrsfs}
\usepackage[all]{xy}
\usepackage{tikz}
\usetikzlibrary{calc}
\usetikzlibrary{positioning}
\usepackage{tikz-cd}
\usetikzlibrary{matrix,arrows}

\newtheorem{theorem}{Theorem}[section]
\newtheorem{lemma}[theorem]{Lemma}
\theoremstyle{definition}
\newtheorem{definition}[theorem]{Definition}
\newtheorem{example}[theorem]{Example}

\newtheorem{prop}[theorem]{Proposition}
\newtheorem{mainthm}{Theorem}

\theoremstyle{remark}
\newtheorem{remark}[theorem]{Remark}

\numberwithin{equation}{section}
\newcommand*{\spf}{\mathop{\textrm{Spf}}}
\newcommand*{\Sub}{\mathop{\textrm{Sub}}}
\begin{document}

\title{Power Operations on $K(n-1)$-Localized Morava $E$-theory at Height $n$}


%    author one information
\author{Yifan Wu}
\address{Southern University of Science \& Technology, Shenzhen, China}
\curraddr{Department of Mathematics}
\email{12131236@mail.sustech.edu.cn}
\thanks{}

%    author two information
%\author{Yifan WU}
%\address{Southern University of Science \& Technology, Shenzhen, China}
%\curraddr{}
%\email{}
%\thanks{}

%    \subjclass is required.
\subjclass[2010]{Primary }

\date{\today}

\dedicatory{}

\keywords{Chromatic homotopy theory, Power operations, Elliptic curves}

%    Abstract is required.
\begin{abstract}
    We calculate the $K(n-1)$-localized $E_n$ theory for symmetric groups, and deduce a modular interpretation of the total power operation $\psi^p_F$ on $F=L_{K(n-1)}E_n$ in terms of augmented deformations of formal groups and their subgroups. We compute the Dyer-Lashof algebra structure over $K(n-1)$-local $E_n$-algebra. Then we specify our calculation to the $n=2$ case. We calculate an explicit formula for $\psi^p_F$ using the formula of $\psi^p_E$, and explain connections between these computations and elliptic curves, modular forms and $p$-divisible groups.
\end{abstract}

\maketitle
\tableofcontents

\bibliographystyle{alpha}
\section{Introduction}
Cohomology operations are crucial in algebraic topology. It equips cohomology rings with more richer algebraic structures. Many non-trivial results have roots in these operations. For instance, the \textit{Steenrod mod-p operations} in mod$-p$ cohomology, the \textit{Adams operations} in the topological $K-$theory are related directly to the \textit{Hopf invariant one} problem\cite{adamshopf,khopf} and the image of the stable $J$-homomorphism\cite{ADAMS196621}. These operations are examples of \textit{power operations}.
\par
Let $E$ be a multiplicative cohomology theory, or equivalently, a homotopy commutative ring spectrum. Suppose $E$ admits an $E_\infty$ structure\cite[Chapter 1,2]{elmendorf1997rings}, i.e., its multiplication structure $E\wedge E\xrightarrow{\mu} E$ commutes not only up to homotopy, but up to higher homotopy coherence. In this case, one can define the \textit{total power operation}
\begin{equation*}
    P_n:E^0\longrightarrow E^0B\Sigma_n,
\end{equation*}
where $B\Sigma_n$ is the classifying space of the $n-$letters permutation groups. See Section \ref{s2.2} for detailed constructions. Some of the many important applications of power operations can be found in \cite{LawsonBP,barthel2024rationalization,balderrama2024total} etc.
\par
It has been long for algebraic topologists to look for correspondences of their topological objects in algebraic geometry since Quillen's work on the connection between complex oriented spectra and formal groups \cite{quillen1969formal}. \textit{Chromatic homotopy theory} provides a framework for systematically study this relationship.
\par
Over an algebraically closed field, formal groups are classified by heights. For each $p$ and height $n$, there is a spectrum $K(n)$, called Morava $K$-theory, which corresponds to the height $n$ formal group \cite{johnson1975bp}. There is also a Morava $E$-theory which parameterizes deformations of such a formal group. Morava $E$ theories admit essentially unique $E_\infty$ structure, hence admit power operations. Surprisingly, the total power operation over it has a strong connection with the moduli of subgroups of formal groups and subgroups of elliptic curves, and it has been well studied by Ando \cite{ando2001elliptic}, Hopkins \cite{ando2004sigma}, Strickland \cite{strickland1998morava} and Rezk\cite{rezk2009congruence},etc.
\par
For a spectrum $X$, the chromatic fracture square and the chromatic convergence theorem suggest that one can break $X$ into pieces, $L_{K(n)}X$ namely, lying in each chromatic layer and recover itself via patching all these pieces together, at least in good circumstance. \textit{Transchromatic homotopy theory} studies such chromatic layers. A fundamental and vital object in transchromatic homotopy theory is the $K(t)-$localized Morava $E-$theory $L_{K(t)}E_n$ for $t\leq n$.
\par
Various work has been devoted to the study of $L_{K(t)}E_n$. In \cite{stapleton2013transchromatic}, Stapleton constructed associated character theory over it using $p-$divisible groups. He and Schlank also gave a transchromatic proof of Strickland's theorem based on such characters and inertia groupoid functors \cite{SCHLANK20151415}. The spectrum $L_{K(t)}E_n$ itself is quite complicated. The coefficient ring $\pi_0L_{K(t)}E_n$ is obtained by first inverting $u_t$ in $$\pi_0E_n=W(k)[\![u_1,\dots,u_{n-1}]\!]$$ then partially completing with the ideal $(p,u_1,\dots,u_{t-1})$. This ring is known to be excellent \cite{barthel2018excellent}. Various different topology could be defined over it. From this point of view, Mazel-Gee, Peterson and Stapleton proposed a modular interpretation of $\pi_*L_{K(t)}E_n$ in terms of pipe rings and pipe formal groups \cite{mazel2015relative}. While when $t=n-1$, things are slightly easier, $\pi_0L_{K(n-1)}E_n$ is still a complete local Noetherian ring with an imperfect residue field $k(\!(u_{n-1})\!)$. Torii compared $L_{K(n-1)}E_n$ with $E_{n-1}$ by studying the associated formal groups \cite{torii2003degeneration,torii2010comparison} and Vankoughnett gave a modular interpretation of $\pi_0L_{K(n-1)}E_n$ using augmented deformations \cite{vankoughnett2021localizations}, which is basically deformations of formal groups together with a choice of the last Lubin-Tate coordinates. (See Section \ref{s2.3} for definitions.)
\par
Motivated by Vankoughnett's result, in this paper, we investigate the modular interpretation of the total power operation on the $K(n-1)$ -localized Morava $E$-theory $L_{K(n-1)}E_n$ at height $n$. 
\par
Let $F=L_{K(n-1)}E_n$, and $\mathbb{G}_F$ be the associated formal group. Let $\mathbb{G}_F^0$ be the fiber of $\mathbb{G}_F$ over the residue field $k(\!(u_{n-1})\!)$ of $F^0$. We showed that
\begin{mainthm}[Theorem \ref{t2.12}]\label{A}
    \textit{The ring $R_m=F^0B\Sigma_{p^m}/I$ classifies augmented deformations of $\mathbb{G}_F^0$ together with a degree $p^m$ subgroup, which means for any complete local Noetherian ring $R$, we have a bijection}
    \begin{equation*}
        {\rm Map}_{cts}(R_m,R)=\{(\mathbb{K},H)\}
    \end{equation*}
    \textit{between the set of continuous maps from $R_m$ to $R$ and the set of pairs consisting of an augmented deformation $\mathbb{K}$ of $\mathbb{G}_F^0$ over $R$ and a degree $p^m$ subgroup $H$ of $\mathbb{K}$. }
    \par
    \textit{Moreover, the additive total power operation}
    \begin{equation*}
        \psi^p_F:F^0\longrightarrow F^0B\Sigma_p/I
    \end{equation*}
    \textit{behaves like take the quotient, i.e.}
    \begin{align*}
        {\rm Map}_{cts}(R_m,R)&\xrightarrow{(\psi_F^p)^*} {\rm Map}_{cts}(F^0,R) \\
        (\mathbb{K},H) &\longmapsto \mathbb{K}/H.
    \end{align*}
\end{mainthm}
\par
Our result is an attempt to understand the full picture of power operations in chromatic setting by studying power operations on each associated $K(n)$-local monochromatic layer. The author also guesses that the algebraic information about the choice of the last Lubin-Tate coordinates has its own modular interpretation in terms of the \'etale part of $G_E$, when considered as a $p$-divisible group over $\pi_0L_{K(n-1)}E_n$.\par
Let $R$ be an $E_\infty$ ring. The homotopy group of an $R$-algebra $A$ possesses a module structure over the \textit{Dyer-Lashof} algebra ${\rm DL}_R$. The Dyer-Lashof algebra is a generalization of Steenrod algebra in generalized cohomology setting, which governs all homotopy operations.\par
Our second result concerns Dyer-Lashof algebra over $K(n-1)$-local $E_n$-algebras. Using Tate spectrum, we show the $L_{K(n-1)}E_n$ theory for symmetric groups is self-dual which leads to the following.
\begin{mainthm}[Proposition \ref{p2.13}]\label{B}
    The Dyer-Lashof algebra over $K(n-1)$-lcaol $E_n$-algebras is 
    \begin{equation*}
        {\rm DL}_{LE}=\bigoplus_{m\geq1} \widehat{F}_0B\Sigma_m
    \end{equation*}
    where $F=L_{K(n-1)}E_n$. Moreover we have
    \begin{equation*}
        \widehat{F}_0B\Sigma_m=F^0B\Sigma_m
    \end{equation*}
    and thus we have
    \begin{equation*}
        \widehat{F}_0B\Sigma_m={\rm Hom}_{F^0}(F^0B\Sigma_m,F^0)
    \end{equation*}
    is the dual of $F^0B\Sigma_m$. The \textit{primitives} in the left hand side correspond to \textit{indecomposables} in the right hand side. 
\end{mainthm}
This allows us to find a presentation of ${\rm DL}_{LE}$ in terms of coefficients of the total power operation $\psi_F^p$, as what have been done in the $K(n)$-local $E_n$-algebra setting.\par
When the height $n=2$, there is a connection between the spectrum $L_{K(1)}E_2$ and elliptic curves. Recall that there is a sheaf of $E_\infty$ rings $\mathcal{O}^{top}$ defined over the \'etale site $(\overline{\mathcal{M}}_{ell})_{\acute{e}t}$, which assigns each elliptic curve
\begin{equation*}
    C: {\rm Spec}(R)\longrightarrow \overline{\mathcal{M}}_{ell}
\end{equation*}
an $E_\infty$ ring $E^C_R$ \cite{MR1989190}. The $p$-completed stack $(\overline{\mathcal{M}}_{ell})_p$ can be decomposed into supersingular part $\overline{\mathcal{M}}_{ell}^{ss}$ and ordinary part $\overline{\mathcal{M}}_{ell}^{ord}$. Hence we have a chromatic fracture square for $p$-completed $Tmf$ \cite{behrens2020topological}:
\begin{equation*}
    \begin{tikzcd}
        Tmf_p \ar[r] \ar[d] & L_{K(2)}Tmf \ar[d] \\
        L_{K(1)}Tmf \ar[r] & L_{K(1)}L_{K(2)}Tmf.
    \end{tikzcd}
\end{equation*}
In the above diagram, the right up corner $L_{K(2)}Tmf$ is a variation of Morava $E$-theories of height 2 and the right lower corner is thus a kind of $K(1)$-localized $E_2$. The spectrum $L_{K(1)}E_2$ can be viewed as the intersection of $\mathcal{O}^{top}$ over $\overline{\mathcal{M}}_{ell}^{ss}$ and $\overline{\mathcal{M}}_{ell}^{ord}$, which corresponds to a punctured formal neighborhood of a supersingular point, as illustrated in the following picture.
\begin{figure}[h]
    \centering
    \includegraphics[width=0.5\linewidth]{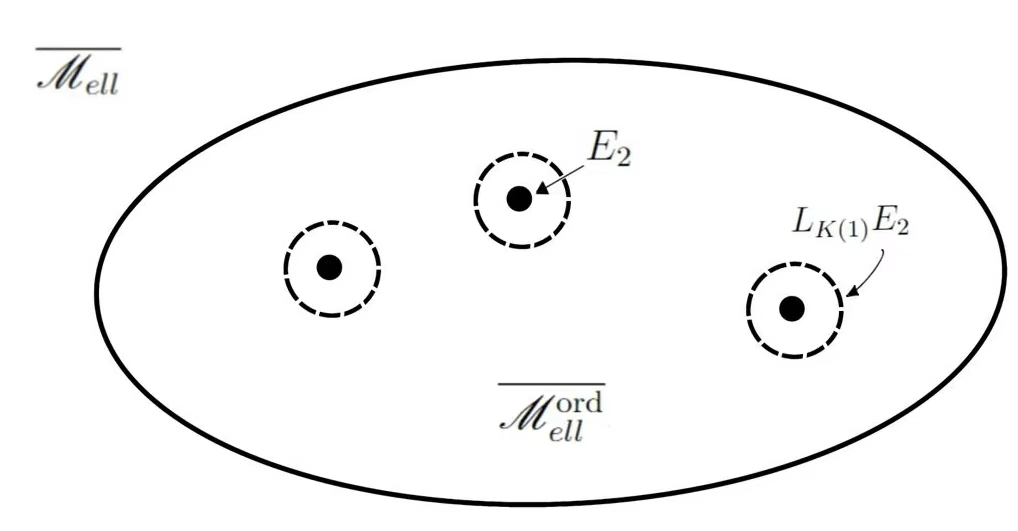}
    \caption{Spectra on $p$ completed stack $\overline{\mathcal{M}}_{ell}$}
    \label{fig:enter-label}
\end{figure}
\par
Our third result concerns the total power operation on $L_{K(1)}E_2$. The Dyer-Lashof algebra structure over $L_{K(1)}E_2$ is clear. It is a free $\theta-$algebra on one generator as stated in \cite{MR3328537}. Via the naturality of total power operations, we obtained the explicit formula of the total power operation on $L_{K(1)}E_2$.
\begin{mainthm}[Theorem \ref{thm3.5}]\label{C}
    \textit{Let $F$ be a $K(1)$-local Morava $E$-theory at height 2. The total power operation $\psi^p_F$ on $F^0$ is determined by}
    \begin{equation}\label{e1.1}
        \psi^p_F(h)=\alpha^*+\sum_{i=0}^p(\alpha^*)^i\sum_{\tau=1}^pw_{\tau+1}d_{i,\tau},
    \end{equation}
    \textit{where }
    \begin{equation}\label{e1.2}
        \alpha^*=(-1)^{p+1}p\cdot h^{-1}+\left(1+(-1)^{p+1}\frac{p(p-1)}{2}\right)p^3\cdot h^{-3}+lower\ terms
    \end{equation}
    \textit{is the unique solution of} 
    \begin{equation*}
        w(h,\alpha)=(\alpha-p)(\alpha+(-1)^p)^p-(h-p^2+(-1)^p)\alpha
    \end{equation*}
    \textit{in} $W(\overline{{\mathbb{F}}}_p)(\!(h)\!)^\wedge_p\cong F^0$.
    \par
    \vspace{\baselineskip}
    \textit{The other coefficients $w_i$ and $d_{i,\tau}$ are defined as}
    \begin{equation*}
        w_i=(-1)^{p(p-i+1)}\left[\binom{p}{i-1}+(-1)^{p+1}p\binom{p}{i}\right]
    \end{equation*}
    \textit{and}
    \begin{equation*}
        d_{i,\tau}=\sum\limits_{n=0}^{\tau-1}(-1)^{\tau-n}w_0^n\sum\limits_{
        \substack{m_1+\cdots m_{\tau-n}=\tau+i\\ 1\leq m_s\leq p+1\\ m_{\tau-n}\geq i+1}}w_{m_1}\cdots w_{m_{\tau-n}}.
    \end{equation*}
    \par
    \textit{In particular, $\psi^p_F$ satisfies the Frobenius congruence, i.e. $$\psi^p_F(h)\equiv h^p\ {\rm mod}\, p$$.}
\end{mainthm}
\begin{remark}
    The element $\alpha^*$ in \ref{e1.2} is the restriction of a modular form $\alpha$ of $\Gamma_0(p)$ over the punctured formal neighborhood of a supersingular point, which parametrizes subgroups of elliptic curves, and $h$ is a lift of Hasse invariant. 
    \par
    The modular form $\alpha$ is sometimes called a \textit{norm parameter} for $\Gamma_0(p)$ \cite[7.5.2]{katz1985arithmetic}\cite[Section 4.3]{ando2000power}. Choosing a coordinate $u$ on elliptic curves with a $\Gamma_0(p)$ structure, i.e. a degree $p$ subgroup, one has
    \begin{equation*}
        \alpha:=\prod_{Q\in\mathscr{G}^{(p)}-O}u(Q)
    \end{equation*}
    where $\mathscr{G}^{(p)}$ is a degree $p$ subgroup. Note that this implies $\alpha$ depends on the choice of $u$. (Remark \ref{r1.2})
    \par
    The element $\psi_F^p(h)$ is actually the image of $h$ under the \textit{Atkin-Lehner involution}, see Section \ref{s3.2} and \cite[Chapter 11]{katz1985arithmetic} for details.
\end{remark}
\begin{remark}\label{r1.2}
    Our computation depends on a specific model for Morava $E$-theories \cite[Definition 2.23]{zhu2019semistable}. The extent of this dependence can be found in \cite[Remark 2.25]{zhu2019semistable}.
\end{remark}
\par
This result can be viewed as a first step toward to the total power operation on the $p$-completed $Tmf$. Our analysis fits in the boxed regions in the diagram below \cite[Page 3]{zhu2019semistable}.

\begin{equation*}
    \begin{tikzpicture}[node distance=2.3cm]
        % Nodes
        \node (ELG) {\footnotesize moduli of $\mathcal{E\ell\ell}_{G}$};
        \node (QELG) [below=of ELG] {\footnotesize moduli of $\mathcal{Q}uasi\mathcal{E}\ell\ell_G$};
        \node (Gamma) [right=of ELG]{\footnotesize moduli of $\mathcal{E\ell\ell}_{\Gamma_1(N)}$};
        \node (E) [right=of QELG] [rectangle,draw] {\footnotesize moduli of $E$};
        \node (C) [right=of Gamma] {\footnotesize moduli of $\mathscr{C}_N$};
        \node (G) [below=of C] [rectangle,draw] {\footnotesize moduli of $\mathbb{G}$};
        \node (L) [below=of E] [rectangle,draw]{\footnotesize moduli of $L_{K(1)}E$};
        \node (A) [below=of G] [rectangle,draw] {\footnotesize moduli of $\widetilde{\mathbb{G}}$};
        \node (P) at ($(QELG)!0.5!(L)$){\footnotesize has power operations};
        % Arrows
        %\draw[->] (ELG) -- node[above] {$G=\Gamma_{1}(N)$} (ELGamma);
        \draw[->] ($(ELG)+(0em,-1.2em)$) -- node[left]{\tiny at cusps} (QELG);
        \draw[->] (ELG) -- node[above]{\tiny $G=\Gamma_1(N)$} (Gamma);
        \draw[->] (Gamma) -- (E);
        \draw[->] (Gamma) -- node[above]{\tiny  derived version of} (C);
        \draw[<->] (C) -- node[sloped,above]{\tiny  Zhu} (E);
        \draw[<->] (C) -- node[right]{\tiny  Serre-Tate} node[left]{\tiny at a s.s point} (G);
        \draw[<->] (E) -- node[above]{\tiny  Ando-Hopkins-Strickland} node[below]{\tiny  Rezk} (G);
        \draw[->] (E) -- node[right]{\tiny $K(1)$-localization} (L);
        \draw[->] (G) -- node[left]{\tiny base change} (A);
        \draw[<->] (L) -- (A);
        \draw[double,->] (P) -- node[sloped,above]{\tiny Huan} (QELG);
        \draw[double,->] (P) -- (E);
        \draw[double,->] (P) -- (L);
        % Additional labels
        \node at (ELG) [below=2pt] {\tiny to be understood as $G$ varies};
        \node at (L) [below=7pt] {\tiny punctured formal neighborhood};
        \node at (L) [below=13pt] {\tiny of a supersingular point};
    \end{tikzpicture}   
\end{equation*}
Here $E$ is a height 2 Morava $E$-theory, $\mathscr{C}_N$ is a universal elliptic curve with an $N$ torsion point and $\mathbb{G}$ is the universal deformation of the associated height 2 formal group of $\mathscr{C}_N$ at a supersingular point.
\subsection{Outline of the paper}
In section \ref{s2.1}, we calculate the $K(n-1)$-localized Morava $E$-theory of symmetric groups using the generalized character map. Then we deduce the modular interpretation of the total power operation $\psi_F^p$ in section \ref{s2.2}.\par
In \cite{ando2004sigma}, the similar moduli interpretation of modified power operations in terms of level structures associated to \textit{abelian} group is confirmed. Though it is claimed in \cite[Remark 3.12]{ando2004sigma} that it is not necessary to use abelian group, there is not a direct proof in nonabelian cases. Section \ref{s2.2} is devoted to such a proof. Seasoned readers can take it for granted and skip it.\par
In section \ref{s2.3} we combine our analysis with augmented deformations and obtain the Theorem \ref{A}.\par
In section \ref{s2.4}, we show the self-dualness of $L_{K(n-1)}E_n$ and compute the Dyer-Lashof algebra structure on $K(n-1)$-local $E_n$-algebras, which is Theorem \ref{B}.\par
In section \ref{s3}, we focus on the particular case for $n=2$. We calculate an explicit formula for the total power operation using the naturality of power operations in section \ref{s3.1} and explain how these ideas are related to elliptic curve, modular forms and $p$-divisible groups in section \ref{s3.2}.\par
The section \ref{s4} is somehow independent from the main line. We investigate a family of spectra, so called augmented deformation spectra and show that there underlying spectra are independent of the choices of formal groups.
\subsection{Acknowledgements}
The author would like to thank Yifei Zhu for his help and guidance during my Ph.D. The author would also like to thank Yingxin Li, Xuecai Ma, Jiacheng Liang, Qingrui Qu and Zhouli Xu for helpful discussions, encouragements and corrections.
\section{$K(n-1)$-localized $E$-theory for symmetric groups}\label{s2}
Let $E$ be the Morava E-theory associated to a height $n$ formal group over a perfect field $k$ with
\begin{equation*}
    \pi_*E=W(k)[\![u_1,\dots,u_{n-1}]\!][u^\pm]
\end{equation*}
and $F$ be the $K(n-1)$-localization of $E$. The coefficient ring 
\begin{equation*}
    \pi_*F=W(k)(\!(u_{n-1})\!)^\wedge_p[\![u_1,\dots,u_{n-2}]\!][u^\pm]
\end{equation*}
is a Noetherian complete local ring with the maximal ideal $(p,u_1,\dots,u_{n-2})$. It satisfies the conditions in \cite[Section 1.3]{hopkins2000generalized}, in particular, $p^{-1}F^*\neq 0$ by direct computation.\par

\subsection{Calculations of $F^*B\Sigma_k$ and $F^*B\Sigma_k/I$}\label{s2.1} 
\begin{theorem}\label{t2.1}
    $F^0B\Sigma_k$ is a Noetherian local ring and a free module over $F^0$ of rank $d(k,n-1)$, which is defined to be the number of isomorphism classes of order $k$ sets with an action of $\mathbb{Z}_p^{n-1}$.
\end{theorem}
\begin{prop}\label{p2.2}
    $F^*B\Sigma_k$ is finitely generated over $F^*$.
\end{prop}
\begin{proof}
    This is a consequence of \cite[Corollary 4.4]{Greenlees1999VARIETIESAL}. We need to verify $F$ is admissible in the sense of \cite[Definition 2.1]{Greenlees1999VARIETIESAL}. $E^0$ is Noetherian and both localization and completion preserve Noetherianess. Hence $F^0$ is Noetherian and all other conditions are satisfied automatically.
\end{proof}
\begin{prop}\label{p2.3}
    $F^*B\Sigma_k$ is free over $F^*$, concentrated in even degrees. 
\end{prop}    
\begin{proof}
    From \cite[Proposition 3.6]{strickland1998morava}, we know that $E^*BG$ is concentrated in even degrees. Let $u_{n-1}^{-1}E$ be the homotopy colimit $$u_{n-1}^{-1}E=E\xrightarrow{u_{n-1}}E\xrightarrow{u_{n-1}}E\rightarrow\cdots$$ where $u_{n-1}$ is the corresponding element in $E^0$ and let $$K_{u_{n-1}}=u_{n-1}^{-1}E/(p,u_1,\dots,u_{n-2})$$ be the successively cofiber, with \begin{equation*}
        \pi_*K_{u_{n-1}}=k(\!(u_{n-1})\!)[u^\pm]
    \end{equation*}\par 
    We claim that $K^*_{u_{n-1}}B\Sigma_k$ is concentrated in even degrees and free. Since $\pi_*K_{u_{n-1}}$ is a graded field $k(\!(u_{n-1})\!)[u^\pm]$, $K_{u_{n-1}}^*B\Sigma_k$ is automatically free. In \cite[Section 7]{hopkins2000generalized}, it has been shown that $\Sigma_k$ is a good group respect to Morava $K$-theory. The argument is still valid if one replaces $K(n)$ with any even periodic field spectrum, which implies our claim.\par
    Now let $$F_i=F/(p,u_1,\dots,u_{i-1})$$ and let $F_0=F$. By construction, we have $F_{n-1}=K_{u_{n-1}}$.\par
    We will show that if $F_i^*B\Sigma_k$ is free and concentrated in even degrees, the same is true for $i-1$ as well. Consider the long exact sequence of cohomology groups
    \begin{equation*}
        F_{i-1}^*B\Sigma_k\rightarrow F_{i-1}^*B\Sigma_k\rightarrow F_{i}^*B\Sigma_k
    \end{equation*}
    obtained from the cofibration
    \begin{equation*}
        F_{i-1}\xrightarrow{u_i}F_{i-1}\rightarrow F_i.
    \end{equation*}
    Each $F_{i}^*B\Sigma_k$ is finitely generated by Proposition \ref{p2.2}. Since $F_i^*B\Sigma_k$ is concentrated in even degrees, multiplying $u_i$ on $F_{i-1}^{\rm odd}B\Sigma_k$ is a surjective. Hence by Nakayama's lemma, $F_{i-1}^{\rm odd}B\Sigma_k=0$. From this, we know the action of $u_i$ on $F_{i-1}^{\rm even}B\Sigma_k$ is regular, and $$F_{i-1}^*B\Sigma_k/u_i=F_{i}^*B\Sigma_k$$ which implies that $F^*_{i-1}B\Sigma_k$ is a free $F^*_{i-1}$ module. Note in particular, we have shown that 
    \begin{equation*}
        K_{u_{n-1}}^*B\Sigma_k=K_{u_{n-1}}^*\otimes_{F^*}F^*B\Sigma_k.
    \end{equation*}
\end{proof}
\begin{proof}[Proof of Theorem 1.1]
    Applying \cite[Theorem C]{hopkins2000generalized}, we have the rank of $$p^{-1}F^*B\Sigma_k$$ over $p^{-1}F^*$ is just $d(n-1,k)$. By Proposition \ref{p2.3}, this rank must equal to the rank of $F^*B\Sigma_k$ over $F^*$.
\end{proof}
\begin{prop}
    The ring $F^0B\Sigma_k/I=0$ for $k\neq p^m$ and $R_m:=F^0B\Sigma_{p^m}/I$ is a free module over $F^0$ of rank $\overline{d}(m,n-1)$, where $I$ is the transfer ideal and $$\overline{d}(m,n-1)=\prod_{t=1}^{n-2}\frac{p^{m+t}-1}{p^t-1}$$ is the number of lattices of index $p^m$ in $\mathbb{Z}_p^{n-1}$.
\end{prop}
\begin{proof}
    For the first sentence, there is a standard argument in \cite[Lemma 8.10]{strickland1998morava}. For the second, using the method in \cite{strickland1997rational} we see that $L(DS^0):=\prod L\otimes_{F^0} F^0B\Sigma_k$ is a Hopf ring, which can be identified with the ring of functions $F(\mathbb{B},L)$, where $L$ is a ring extension of $F^0$ with $p^{-1}$ and all roots of the $p$-series of the formal group law over $F^0$ added and $\mathbb{B}$ is the Burnside semiring.\par 
    The $\times-$indecomposables ${\rm Ind}L(DS^0)=\prod L\otimes_{F^0}F^0B\Sigma_k/I_k$ is identified with $F(\mathbb{L},L)$, where $\mathbb{L}$ is the set if all lattices in $\mathbb{Z}_p^{n-1}$ and $I_k$ is the transfer. Hence we have an isomorphism $L\otimes_{F^0}F^0B\Sigma_k/I_k\cong F(\mathbb{L}_k,L)$, with $\mathbb{L}_k$ being the set of such lattices of index $k$. This implies the rank of $R_m$ over $F^0$ is $\overline{d}(m,n-1)$.
\end{proof}

\subsection{Modular interpretation of $\psi^p_F$}\label{s2.2} Let us first recall the construction of the additive total power operation.\par
Let $f\in E^0(X)$, which is represented by a map
\begin{equation*}
    f:X\longrightarrow E.
\end{equation*}
Then we obtained the $n$th power $f^n$ via the composition
\begin{equation*}
    X\xrightarrow{\Delta}\underbrace{X\wedge\cdots\wedge X}_n\xrightarrow{f^{\wedge n}}\underbrace{E\wedge\cdots\wedge E}_n\xrightarrow{\mu}E.
\end{equation*}
It is clear that the symmetric group $\Sigma_n$ acts on this map and preserves it. Hence it factors as 
\begin{equation*}
    f^n:X\rightarrow X\times B\Sigma_n\rightarrow E^{\wedge n}_{h\Sigma_n}\xrightarrow{\mu}E,
\end{equation*}
where the first map is including the base point. The latter composition is an element in $E^0(X\times B\Sigma_n)$, denoted by $\psi^n(f)$. Thus we have obtained a refined $n$th power operation
\begin{equation*}
    \Psi^n:E^0(X)\longrightarrow E^0(X\times B\Sigma_n)
\end{equation*}
called \textit{the nth total power operation}. Composing with including the base point of $B\Sigma_n$ yields the ordinary $n$th power operation. Note that $\Psi^n$ is only multiplicative. To obtain a ring map, we can further mod out non-additive terms:
\begin{equation*}
    \psi^n:E^0(X)\longrightarrow E^0(X\times B\Sigma_n)\longrightarrow E^0(X\times B\Sigma_n)/I
\end{equation*}
where $I$ is the transfer ideal generated by the images of the transfer maps
\begin{equation}
    {\rm Tr}: E^0\left(X\times(B\Sigma_i\times B\Sigma_j)\right)\rightarrow E^0(X\times B\Sigma_n) 
\end{equation}
for all $i+j=n$,\cite[Section 11.3]{rezk2006lectures}.\par

Let $\mathbb{G}_E$ and $\mathbb{G}_F$ be the formal groups over $\spf(E^0)$ and $\spf(F^0)$ respectively. In \cite[Section 9]{strickland1998morava}, the scheme $\spf(E^0B\Sigma_{p^k}/I)$ is identified with the subgroup scheme $\Sub_m(\mathbb{G}_E)$ \cite[Theorem 10.1]{strickland1997finite} over $\spf(E^0)$.\par
The same procedure can be carried through with $E$ replaced by $F$ without harm.
\begin{prop}\label{p2.5}
    There is a canonical isomorphism $\spf(F^0B\Sigma_{p^m}/I)\rightarrow\Sub_m(\mathbb{G}_F)$. That is, the ring $F^0B\Sigma_{p^m}/I$ classifies degree $p^m$ subgroups of $\mathbb{G}_F$.
\end{prop}
\begin{proof}
    There is a canonical map from $\mathcal{O}_{\Sub_m(\mathbb{G}_F)}$ to $F^0B\Sigma_{p^m}/I$ as constructed in \cite[Proposition 9.1]{strickland1998morava}. Note that, these two rings has the same rank over $F^0$. So we proceed as \cite[Theorem 9.2]{strickland1998morava}, by showing 
    \begin{equation*}
        k(\!(u_{n-1})\!)\otimes_{F^0}\mathcal{O}_{\Sub_m(\mathbb{G}_F)}\rightarrow k(\!(u_{n-1})\!)\otimes_{F^0}F^0B\Sigma_{p^m}/I
    \end{equation*}
    is injective. The key ingredient here is to show $b_m=c_{p^m}^{(p^{n-1}-1)/(p-1)}\neq 0$ in $$k(\!(u_{n-1})\!)\otimes_{F^0}F^0B\Sigma_{p^m}=K_{u_{n-1}}^0B\Sigma_{p^m},$$ where $c_{p^m}=e(V_{p^m}-1)$ is the Euler class of representation $V_{p^m}-1$ in $F^0B\Sigma_{p^m}$ and $V_{p^m}$ is the standard complex representation of $\Sigma_{p^m}$. This follows from \cite[Theorem 3.2]{strickland1998morava} with $K$ replaced by $K_{u_{n-1}}$. The rest follows \cite[Theorem 9.2]{strickland1998morava}.
\end{proof}
\begin{remark}
    We can not obtain this result directly from \cite[Theorem 10.1]{strickland1997finite} which asserts that \begin{equation*}
    \spf F^0\times_{\spf E^0}\Sub\nolimits_m(\mathbb{G}_E)=\Sub\nolimits_m(\spf F^0\times_{\spf E^0}\mathbb{G}_E)=\Sub\nolimits_m(\mathbb{G}_F).
    \end{equation*}
    The failure of this equation is because the map $E^0\rightarrow F^0$ is not continuous.
\end{remark}\par
In order to figure out how the total power operation 
\begin{equation*}
    \psi^p_F:F^0\longrightarrow F^0B\Sigma_p/I
\end{equation*}
interacts with the modular interpretation of $F^0B\Sigma_p/I$, we shall recall some constructions from \cite[Section 3]{ando2004sigma}.\par
Let $Y$ denote the function spectrum $F(\mathbb{C}P^\infty, F)$, we have 
\begin{equation*}
    \pi_0Y=F^0\mathbb{C}P^\infty=F^0[\![x]\!]
\end{equation*}
which is a complete local Noetherian ring, with maximal ideal $(p,u_1,\dots,u_{n-2},x)$ and the canonical map $\pi_0F\rightarrow\pi_0Y$ is continuous with respect to their maximal ideal topology.
\begin{prop}\label{1.7}
    The ring $Y^0B\Sigma_p/J$ is free over $Y^0$ and equal to $Y^0\otimes_{F^0}F^0B\Sigma_p/I$, where $I$ and $J$ are transfer ideals respectively.
\end{prop}
\begin{proof}
    For each $k$, we have
    \begin{equation*}
        Y^*B\Sigma_k=[\Sigma_+^\infty B\Sigma_k, F(\mathbb{C}P^\infty, F)]=[\Sigma_+^\infty (B\Sigma_k\wedge\mathbb{C}P^\infty),F]=F^*(B\Sigma_k\wedge\mathbb{C}P^\infty).
    \end{equation*}
    By the Atiyah Hirzebruch spectral sequence, we have
    \begin{equation*}
        E_2^{p,q}=H^p(\mathbb{C}P^\infty, F^qB\Sigma_k)\Rightarrow Y^{p+q}B\Sigma_k
    \end{equation*}
    Since $F^*B\Sigma_k$ is concentrated in even degrees, we conclude that $$Y^*B\Sigma_k=Y^*\otimes_{F^*}F^*B\Sigma_k.$$
    It follows that $Y^0\otimes_{F^0}I=J$, and hence 
    \begin{equation*}
        Y^0B\Sigma_p/J=Y^0\otimes_{F^0}F^0B\Sigma_p/I.
    \end{equation*}
    which completes the proof.
\end{proof}
In the language of algebraic geometry, $\spf Y^0=\mathbb{G}_F$ and the above proposition can be summarized as the pullback diagram.
\begin{center}
    \begin{tikzpicture}
        \matrix (m) [
            matrix of math nodes,
            row sep=2.5em,
            column sep=2.5em,
            text height=1.5ex, text depth=0.25ex
        ]
        {    \spf(Y^0B\Sigma_p/J)=i^*\mathbb{G}_F & \mathbb{G}_F \\
             \spf(F^0B\Sigma_p/I) & \spf F^0 \\
        };

        \path[overlay,->, font=\scriptsize,>=latex]
        (m-1-1) edge (m-1-2)
        (m-1-1) edge (m-2-1)
        (m-1-2) edge (m-2-2)
        (m-2-1) edge (m-2-2)
        ;
    \end{tikzpicture}
\end{center}
Together with the naturality of the total power operation:
\begin{center}
    \begin{tikzpicture}
        \matrix (m) [
            matrix of math nodes,
            row sep=2.5em,
            column sep=2.5em,
            text height=1.5ex, text depth=0.25ex
        ]
        {    i^*\mathbb{G}_F & \mathbb{G}_F \\
             \spf(F^0B\Sigma_p/I) & \spf F^0 \\
        };

        \path[overlay,->, font=\scriptsize,>=latex]
        (m-1-1) edge node[auto] {\(\psi_Y^*\)}(m-1-2)
        (m-1-1) edge (m-2-1)
        (m-1-2) edge (m-2-2)
        (m-2-1) edge node[auto] {\(\psi_F^*\)} (m-2-2)
        ;
    \end{tikzpicture}
\end{center}
we obtain a map $\psi^*_{Y/F}:i^*\mathbb{G}_F\rightarrow(\psi_F^p)^*\mathbb{G}_F$ over the ring $F^0B\Sigma_p/I$, as indicated in the diagram.
\begin{center}
    \begin{tikzpicture}
        \matrix (m) [
            matrix of math nodes,
            row sep=2.5em,
            column sep=2.5em,
            text height=1.5ex, text depth=0.25ex
        ]
        {  i^*\mathbb{G}_F & &  \\
            & (\psi_F^p)^*\mathbb{G}_F & \mathbb{G}_F \\
            & \spf(F^0B\Sigma_p/I) & \spf F^0 \\
        };

        \path[overlay,->, font=\scriptsize,>=latex]
        (m-1-1) edge[dashed] node[auto] {$\psi^*_{Y/F}$} (m-2-2)
        (m-1-1) edge[bend left] node[auto] {\(\psi_Y^*\)}(m-2-3)
        (m-1-1) edge[bend right] (m-3-2)
        (m-2-2) edge (m-2-3)
        (m-2-2) edge (m-3-2)
        (m-2-3) edge (m-3-3)
        (m-3-2) edge node[auto] {\((\psi_F^p)^*\)} (m-3-3)
        ;
    \end{tikzpicture}
\end{center}
\begin{prop}\label{p2.8}
    The isogeny $\psi^*_{Y/F}:i^*\mathbb{G}_F\rightarrow(\psi_F^p)^*\mathbb{G}_F$ is of degree $p$ over $F^0B\Sigma_p/I$, with kernel the universal degree $p$ subgroup $K$ of $\mathbb{G}_F$ over $F^0B\Sigma_p/I$.
\end{prop}
\begin{proof}
    Choosing a coordinate $x$ on $\mathbb{G}_F$, $\psi_Y^*$ sends $x$ to $x^p$ in $Y^0B\Sigma_p/J=\mathcal{O}_{i^*\mathbb{G}_F}$ modulo maximal ideal of $Y^0$. This follows from
    \begin{equation*}
        \pi_0Y\xrightarrow{D_p} \pi_0Y^{B\Sigma_p^+} \xrightarrow{S^0\rightarrow B\Sigma_p^+} \pi_0Y
    \end{equation*}
    sending $x$ to $x^p$.
    Since $(\psi^p_F)^*(x)=x$, we conclude that $\psi^*_{Y/F}$ is of degree $p$. Therefore the kernel of $\psi^*_{Y/F}$ is of rank $p$.\par
    To show the kernel is precisely the universal degree $p$ subgroup $K$ of $\mathbb{G}_F$ over $F^0B\Sigma_p/I$, we need to recall the construction of $K$ from \cite[Proposition 9.1]{strickland1998morava}(in which $K$ is denoted by $H_k$).
    Let $V_p$ be the standard permutation representation of $\Sigma_p$. There is a divisor $\mathbb{D}(V_p)$ of degree $p$ over $F^0B\Sigma_p$, whose base change to $F^0B\Sigma_p/I$ is $K$. Let $A$ be a transitive abelian $p$ subgroup of $\Sigma_p$, we have a composition of maps
    \begin{equation*}
        {\rm Level}(A^*,\mathbb{G}_F)\rightarrow{\rm Hom}(A^*,\mathbb{G}_F)=\spf F^0BA\rightarrow \spf F^0B\Sigma_p.
    \end{equation*}
    The divisor $\mathbb{D}(V_p)$ becomes a subgroup divisor $\Sigma_{a\in A^*}[\ell(a)]$ with $\ell$ the universal level-$A^*$ structure of $\mathbb{G}_F$ on ${\rm Level}(A^*,\mathbb{G}_F)$(See \cite[Section 3]{ando2004sigma} for definition).
    It is claimed in \cite[Proposition 9.1]{strickland1998morava} that the map 
    \begin{equation*}
        {\rm Level}(A^*,\mathbb{G}_F)\rightarrow\spf F^0B\Sigma_p
    \end{equation*}
    factors through $\spf F^0B\Sigma_p/I$ and the union of the images of these maps for all such $A$ is actually $\spf F^0B\Sigma_p/I$. Hence it is sufficient to show the base change of $\ker\psi^*_{Y/F}$ to ${\rm Level}(A^*,\mathbb{G}_F)$ is $\Sigma_{a\in A^*}[\ell(a)]$.\par
    Now Let $D(A)=\mathcal{O}_{{\rm Level}(A^*,\mathbb{G}_F)}$, the following diagram
    \begin{center}
        \begin{tikzcd}
            F^0 \ar[r, "D_A"'] \ar[rr, "\psi_F^\ell" description, bend left]  & F^0BA \ar[r] & D(A) \\
            F^0 \ar[u, equal] \ar[rr, "\psi_F^p" description, bend right] \ar[r, "D_p"] & F^0B\Sigma_p \ar[u] \ar[r] & F^0B\Sigma_p/I \ar[u, dashed]
        \end{tikzcd}
    \end{center}
    
    implies the composition of the total power operation $\psi_F^p$ and the dashed arrow is $\psi^\ell_F$ (See \cite[Definition 3.9]{ando2004sigma}). Hence after base change to ${\rm Level}(A^*,\mathbb{G}_F)$, the map $\psi^*_{Y/F}$ becomes $\psi_\ell^{Y/F}$ \cite[diagram 3.14]{ando2004sigma}. According to \cite[Proposition 3.21]{ando2004sigma}, the kernel of $\psi^{Y/F}_\ell$ is precisely $\ell[A]=\Sigma_{a\in A^*}[\ell(a)]$.
\end{proof}
\subsection{Augmented deformations}\label{s2.3}
In this section, we combine our analysis about $F^0B\Sigma_p/I$ and the modular interpretation of $F^0$ in terms of augmented deformations. Recall that there is a formal group $\mathbb{G}_F$ over $F^0$, which is the base change of the universal deformation $\mathbb{G}_E$. Let $\mathbb{G}_F^0$ be the special fiber of $\mathbb{G}_F$, which is the base change of $\mathbb{G}_F$ over the residue field $k(\!(u_{n-1})\!)$ of $F^0$.\par
The formal group $\mathbb{G}_F^0$ has height $n-1$ over $k(\!(u_{n-1})\!)$. At first glance, one would like to construct the deformation theory of $\mathbb{G}_F^0$ as \cite{lubin1966formal} does. However, the problem arises immediately for the field $k(\!(u_{n-1})\!)$ being imperfect. A way to avoid the imperfectness is the treatment stated in \cite{vankoughnett2021localizations}. We shall recall these constructions.
\begin{definition}
    An augmented deformation of a formal group $\mathbb{H}$ over $k(\!(u_{n-1})\!)$ consists of a triple $(\mathbb{K}/R,i,\alpha)$ where
    \begin{itemize}
        \item $R$ is a complete local ring and $\mathbb{K}$ is a formal group over $R$,
        \item A local homomorphism $i:\Lambda\rightarrow R$ fits into the commutative diagram
        \begin{center}
            \begin{tikzcd}
                \Lambda \ar[r,"i"] \ar[d]  & R \ar[d] \\
                k(\!(u_{n-1})\!) \ar[r, " \Bar{i}"] & R/\mathfrak{m}
            \end{tikzcd}
        \end{center}
        \item and an isomorphism $\alpha:\mathbb{H}\otimes^{\overline{i}}_{k(\!(u_{n-1})\!)}R/\mathfrak{m}\simeq\mathbb{K}\otimes_RR/\mathfrak{m}$,
    \end{itemize}
    where $\Lambda=W(k)(\!(u_{n-1})\!)^\wedge_p$ is a Cohen ring with residue field $k(\!(u_{n-1})\!)$.
\end{definition}
\begin{remark}
    There always exists such a local homomorphism $i:\Lambda\rightarrow R$ filling the diagram
    \begin{center}
        \begin{tikzcd}
            \Lambda \ar[dashed, r] \ar[d]  & R \ar[d] \\
            k(\!(u_{n-1})\!) \ar[r] & R/\mathfrak{m}
        \end{tikzcd}
    \end{center}
    due to the property of Cohen rings. Note that such morphisms may not be unique.\par
    This is the main difference between deformation theories over perfect fields and imperfect fields. When the field on the left lower corner is perfect, there is a unique local homomorphism from the Witt ring $W(k)$ to $R$ and consequently a unique $W(k)$-algebra structure over $R$. While this is not true for Cohen rings of imperfect fields. Hence one must specify a $\Lambda$-algebra structure when discussing deformations in the imperfect context.
\end{remark}
The  moduli problem of classifying augmented deformations is representable by the ring $F^0=\pi_0L_{K(n-1)}E_n$ as described below.
\begin{theorem}[\cite{vankoughnett2021localizations}, Theorem 1.1]\label{t2.11}
    Let $\mathbb{H}$ be any height $n-1$ formal groups over $k(\!(u_{n-1})\!)$. The ring $F^0$ classifies augmented deformations of $\mathbb{H}$. To be precise, let ${\rm Def}^{\rm {aug}}_{\mathbb{H}}(R)$ denote the groupoid of augmented deformations of $\mathbb{H}$ together with isomorphisms. Then we have
    \begin{equation*}
        {\rm Def}^{\rm{aug}}_{\mathbb{H}}(R)={\rm Maps}_{cts}(F^0,R).
    \end{equation*}
    In particular, this implies the moduli problem of classifying augmented deformation is discrete.
\end{theorem}
Combining our previous analysis, namely Proposition \ref{1.7}, \ref{1.8}, we obtain our Theorem \ref{A}.
\begin{theorem}\label{t2.12}
    The ring $F^0B\Sigma_{p^m}/I$ is free over $F^0$ of rank $\overline{d}(m,n-1)$. It classifies augmented deformations of $\mathbb{G}_F^0$ together with a subgroup of degree $p^m$. 
    \begin{equation*}
        {\rm Maps}_{cts}(F^0B\Sigma_{p^m}/I,R)=\{(\mathbb{K}/R, H) \}
    \end{equation*}
    To be precise, for any complete local ring $R$, there is a bijection between the set of continuous maps from $F^0B\Sigma_{p^m}/I$ to $R$ and the set of all pairs $(\mathbb{K}/R,H)$, where $\mathbb{K}$ is an augmented deformations of $\mathbb{G}_F^0$ and $H$ is a degree $p^m$ subgroup of $\mathbb{K}$.\par
    Equivalently, $F^0B\Sigma_{p^m}/I$ classifies augmented deformations of $m$th Frobenius \cite[Section 11.3]{rezk2009congruence}, with the universal example
    \begin{equation*}
        \psi^*_{Y/F}:i^*\mathbb{G}_F\rightarrow(\psi^{p^m}_F)^*\mathbb{G}_F
    \end{equation*}
    defined in the Proposition \ref{p2.8}.
\end{theorem}
\begin{proof}
    Combines Proposition \ref{p2.5}, \ref{p2.8} and Theorem \ref{t2.11}
\end{proof}
\subsection{Dyer-Lashof algebra on $K(n-1)$-local $E$ algebras}\label{s2.4}
Let $R$ be an $E_\infty$ ring. The \textit{Dyer-Lashof} algebra ${\rm DL}_R$ is an algebraic theory, which describes all homotopy operations on $R$-algebras. \cite[Section 9.1]{rezk2006lectures}\par
For simplicity, we restrict ourselves only in degree $0$ part. In this case, the Dyer-Lashof algebra is described by the ring $\oplus_{m\geq 0} R_0B\Sigma_m$. Let $A$ be an $R$-algebra. Choosing an element $\alpha\in R_0B\Sigma_m$, which is represented by an $R$-module map $$R\longrightarrow R\wedge B\Sigma_m$$ we can obtain an \textit{individual operation} from $\pi_0A$ to itself via composing with the total power operation.
\begin{equation*}
    R\xrightarrow{\alpha} R\wedge B\Sigma_m\xrightarrow{P_m} A
\end{equation*}\par
Let $E$ be the Morava $E$-theory of height $n$. The category of $K(n-1)$-local $E$-algebra is equivalent to the category of $K(n-1)$-local $L_{K(n-1)}E$-algebra, because for any $K(n-1)$-local $E$-algebra $A$, the structure map $E\rightarrow A$ factors through
\begin{equation*}
    E\longrightarrow L_{K(n-1)}E\longrightarrow A
\end{equation*}
Thus the calculation of the Dyer-Lashof algebra on $K(n-1)$-local $E$-algebra demands the calculation of $F_0B\Sigma_m$, where $F=L_{K(n-1)}E$, as usual. 
To obtain good values on Dyer-Lashof algebra, we modify our notation with the completed homology
\begin{equation*}
    \widehat{F}_0B\Sigma_m=\pi_0L_{K(n-1)}(F\wedge B\Sigma_m)
\end{equation*}
\begin{prop}\label{p2.13}
    There is an isomorphism
    \begin{equation*}
        \widehat{F}_0B\Sigma_m\rightarrow F^0B\Sigma_m
    \end{equation*}
    Hence the Dyer-Lashof algebra over $K(n-1)$-local $E$-algebra is generated by the coefficients of the total power operation $\psi_F^p$.
\end{prop}
\begin{proof}
    As explained above, the second statement is straightforward if the isomorphism holds. Consider the Greenlees-May's cofiber sequence \cite{greenlees1995generalized}
    \begin{equation*}
        k\wedge EG \rightarrow F(EG,k)\rightarrow t_G(k)
    \end{equation*}
    Taking homotopy pixed point yields a cofiber sequence
    \begin{equation*}
        K\wedge BG\rightarrow F(BG,K)\rightarrow t_G(i_*K)^G
    \end{equation*}
    where $K$ is any spectrum and $k=i_*K$ is the $G$-equivariant version of $K$ \cite[Section 1.1]{greenlees1996tate}.\par
    We will show that the spectrum $t_G(i_*F)^G$ is $K(n-1)$ acyclic and hence we have an equivalence
    \begin{equation*}
        L_{K(n-1)}(F\wedge BG)\rightarrow L_{K(n-1)}F(BG,F)=F(BG,F)
    \end{equation*}
    which induces the desired isomorphism.\par
    Our strategy is as follow. Choose a generalized Moore spectrum $M$ of type $n-1$, such that $M\wedge F$ is generated by $K_{u_{n-1}}$, which is defined in Proposition \ref{2.3}. Then we show 
    \begin{equation}\label{e2.2}
        t_G(i_*K_{u_{n-1}})^G=0
    \end{equation}
    Hence by property of Tate cohomology, we have
    \begin{equation*}
        t_G(i_*(M\wedge F))^G=M\wedge t_G(i_*F)^G=0
    \end{equation*}
    Applying $K(n-1)$ homology and Kunneth formula imples
    \begin{equation*}
        K(n-1)_*(t_G(i_*F)^G)=0
    \end{equation*}
    According to \cite[Proposition 3.1]{greenlees1996tate}, the equation \ref{e2.2} holds if for all finite group $G$, $K_{u_{n-1}*}BG$ is finite generated as $K_{u_{n-1}*}$ module. This is true because $K_{u_{n-1}}$ is admissible and the finiteness follow from \cite[Corollary 4.4]{Greenlees1999VARIETIESAL}, or one can directly check that
    \begin{equation*}
        K_{u_{n-1}*}(-)=K_*(-)\otimes_{K_*}K_{u_{n-1}*}
    \end{equation*}
    and the condition of being finite is valid for $K$, where $K$ is the even periodic version of $K(n)$.
\end{proof}
\section{An explicit Calculation on the $n=2$ case}\label{s3}
Let $E$ be a Morava $E$-theory of height 2 over the field $\overline{\mathbb{F}}_p$, with
\begin{equation*}
    E^*=W(\overline{\mathbb{F}}_p)[\![u_1]\!][u^\pm].
\end{equation*}
Let $F$ be the $K(1)$ localization of $E$, whose coefficients ring is 
\begin{equation*}
    F^*=W(\overline{\mathbb{F}}_p)(\!(u_1)\!)^\wedge_p[u^\pm].
\end{equation*}
Let $\mathbb{G}_E$ and $\mathbb{G}_F$ be the formal groups over $E^0$ and $F^0$ respectively.\par
In this section, we give an explicit calculation of the additive total power operation $\psi_F^p$ in terms of the expression of $\psi^p_E$ for the $n=2$ case.
\subsection{The formula for $\psi^p_F$}\label{s3.1} The naturality of the total power operations gives a diagram:
\begin{equation}\label{e3.1}
    \begin{tikzcd}[column sep=3em,/tikz/column 2/.style={column sep=0.7em}]
        E^0 \ar[r, "\psi^p_E"] \ar[d] & E^0B\Sigma_p/I \ar[d, "t"] & \\
        F^0 \ar[r, "\psi^p_F"]        & F^0B\Sigma_p/J \ar[r, equal] & F^0
    \end{tikzcd}
\end{equation}
where $I$ and $J$ are the corresponding transfer ideals. The equality on the right corner is because the formal group $\mathbb{G}_F$ is of height 1, hence $F^0B\Sigma_p/J$ is free of rank $\Bar{d}(1,1)=1$ over $F^0$.
\begin{remark}
    From now on, we will use $h$ instead of $u_1$ in $E^*$ and $F^*$. This is because when height is 2, the ring $E^0$ can be viewed as the place where the universal deformation of a certain supersingular elliptic curve is defined. The letter $h$ here stands for the Hasse invariant for it being a lift of Hasse invariant.
\end{remark}
The map $t$ in the middle is $E^0$ linear. To see this, consider the diagram
\begin{center}
    \begin{tikzcd}
        E^0(\bigvee_{i=1}^{p-1}B\Sigma_i\times B\Sigma_{p-i}) \ar[r, "tr_E"] \ar[d] & E^0B\Sigma_p \ar[r] \ar[d] & E^0B\Sigma_p/I \ar[d, "t", dashed]\\
        F^0(\bigvee_{i=1}^{p-1}B\Sigma_i\times B\Sigma_{p-i}) \ar[r, "tr_F"] & F^0B\Sigma_p \ar[r] & F^0B\Sigma_p/J
    \end{tikzcd}
\end{center}
The maps in the top row are between $E^0$ modules and maps in the bottom can also be viewed as $E^0$ linear maps via $E^0\rightarrow F^0$. Then one can check that the left two vertical maps are $E^0$ linear, which implies $t$ is $E^0$ linear as well.\par
Now we can deduce the explicit expression of $\psi_F^p$ via the calculation of $\psi^p_E$, which is summarized in the two theorems below.
\begin{theorem}[\cite{zhu2019semistable}, Theorem A]\label{t3.2}
    After choosing a preferred model for $E$ \cite[Definition 2,23]{zhu2019semistable}, the ring $E^0B\Sigma_p/I$ can be interpreted as 
    \begin{equation*}
        E^0B\Sigma_p/I=W(\overline{\mathbb{F}}_p)[\![h,\alpha]\!]/w(h,\alpha)
    \end{equation*}
    with
    \begin{equation}\label{e3.2}
        w(h,\alpha)=(\alpha-p)\left(\alpha+(-1)^p\right)^p-\left(h-p^2+(-1)^p\right)\alpha.
    \end{equation}
\end{theorem}
\begin{theorem}[\cite{zhu2019semistable}, Theorem B]\label{t3.3}
    The image of $h$ under $\psi^p_E$ is 
    \begin{equation}
        \psi^p_E(h)=\alpha+\sum_{i=0}^p\alpha^i\sum_{\tau=1}^pw_{\tau+1}d_{i,\tau},
    \end{equation}
    where $w_i$'s are defined to be
    \begin{equation*}
        w_i=(-1)^{p(p-i+1)}\left[\binom{p}{i-1}+(-1)^{p+1}p\binom{p}{i}\right]
    \end{equation*}
    and
    \begin{equation*}
        d_{i,\tau}=\sum\limits_{n=0}^{\tau-1}(-1)^{\tau-n}w_0^n\sum\limits_{
        \substack{m_1+\cdots m_{\tau-n}=\tau+i\\ 1\leq m_s\leq p+1\\ m_{\tau-n}\geq i+1}}w_{m_1}\cdots w_{m_{\tau-n}}.
    \end{equation*}
\end{theorem}
To determine the image of $h\in F^0=W(\overline{\mathbb{F}}_p)(\!(h)\!)^\wedge_p$ under $\psi^p_F$, it suffices to determine the image of $\alpha$ in Theorem \ref{t3.2} under the map $t$. We have
\begin{equation*}
    \psi^p_F(h)=t\circ\psi^p_E(h)
\end{equation*}
by the diagram \ref{e3.1}. Since $t$ is an $E^0$ linear map, this requires us to find the solutions of $w(h,\alpha)$ in $F^0$.
\begin{prop}\label{p3.4}
    There is a unique solution $\alpha^*$ of $w(h,\alpha)$ in $W(\overline{\mathbb{F}}_p)(\!(h)\!)^\wedge_p$ with
    \begin{equation}
        \alpha^*=(-1)^{p+1}p\cdot h^{-1}+\left(1+(-1)^{p+1}\frac{p(p-1)}{2}\right)p^3\cdot h^{-3}+lower\ terms
    \end{equation}
    satisfies 
    \begin{equation*}
        w(h,\alpha)=(\alpha-p)(\alpha+(-1)^p)^p-(h-p^2+(-1)^p)\alpha=0.
    \end{equation*}
    Moreover, we have $\alpha^*=0$ mod $p$.
\end{prop}
\begin{proof}
    We write $w(h,\alpha)$ as $w_{p+1}\alpha^{p+1}+w_p\alpha^p+\cdots+w_1\alpha+w_0$, where $w_{p+1}=1$, $w_1=-h$, $w_0=(-1)^{p+1}p$, and 
    \begin{equation*}
        w_i=(-1)^{p(p-i+1)}\left[\binom{p}{i-1}+(-1)^{p+1}p\binom{p}{i}\right]
    \end{equation*}
    for other coefficients.
    \par
    Since $h$ is invertible in $W(\overline{\mathbb{F}}_p)(\!(h)\!)^\wedge_p$, the equation $w(h,\alpha)=0$ implies
    \begin{align*}
        \alpha &=h^{-1}(\alpha^{p+1}+w_p\alpha^p+\cdots w_2\alpha^2+w_0)\\
        &=h^{-1}w_0 + \alpha^2(\alpha^{p-1}+w_p\alpha^{p-2}+\cdots+w_2)h^{-1}\\
        &=h^{-1}w_0 + h^{-3}w_0^2w_2 + lower\ terms
    \end{align*}
    Substituting the second equation into itself recursively gives the desired formula for $\alpha^*$ as described in .
    \par
    This iteration makes sense because the highest term of $\alpha^*$ is $h^{-1}w_0$ and $p|w_0$. Hence each substitution only create a lower terms, which is divided by a higher power of $p$, than current stage. Hence $\alpha^*=\Sigma_ka_kh^{-k}$ and the coefficient $a_k$ satisfies $\lim_{k\rightarrow\infty}|a_k|=0$, which implies $\alpha^*$ is indeed an element in $W(\overline{\mathbb{F}}_p)(\!(h)\!)^\wedge_p$.\par
    The uniqueness comes from the following observation. Note that $$w(h,\alpha)=\alpha(\alpha^p-h)\ {\rm mod}\ p.$$ This implies $w(h,\alpha)$ has only one solution $0$ in the residue field of $W(\overline{\mathbb{F}}_p)(\!(h)\!)^\wedge_p$. Therefore it also has a unique solution in $W(\overline{\mathbb{F}}_p)(\!(h)\!)^\wedge_p$, which is $\alpha^*$.
\end{proof}
\begin{theorem}\label{thm3.5}
    Let $F$ be a $K(1)$-local Morava $E$-theory at height 2. The total power operation $\psi^p_F$ on $F^0$ is determined by
    \begin{equation}\label{e3.5}
        \psi^p_F(h)=\alpha^*+\sum_{i=0}^p(\alpha^*)^i\sum_{\tau=1}^pw_{\tau+1}d_{i,\tau},
    \end{equation}
    where 
    \begin{equation*}
        \alpha^*=(-1)^{p+1}p\cdot h^{-1}+\left(1+(-1)^{p+1}\frac{p(p-1)}{2}\right)p^3\cdot h^{-3}+lower\ terms
    \end{equation*}
    is the unique solution of 
    \begin{equation*}
        w(h,\alpha)=(\alpha-p)(\alpha+(-1)^p)^p-(h-p^2+(-1)^p)\alpha
    \end{equation*}
    in $W(\overline{{\mathbb{F}}}_p)(\!(h)\!)^\wedge_p\cong F^0$.
    \par
    The other coefficients $w_i$ and $d_{i,\tau}$ are defined in Theorem \ref{t3.3}.
    \par
    In particular, $\psi^p_F$ satisfies the Frobenius congruence, i.e. $$\psi^p_F(h)\equiv h^p\ {\rm mod}\, p$$.
\end{theorem}
\begin{proof}
    The formula \ref{e3.5} is obtained by assembling Theorem \ref{t3.3} and Proposition \ref{p3.4}. The last sentence comes from $\psi^p_F\equiv \sum_{\tau=1}^pw_{\tau+1}d_{0,\tau}$ {\rm mod} $p$, for $\alpha^*$ being zero after modulo $p$. Also notice that 
    \begin{equation*}
        w_i\equiv 0\ {\rm mod}\ p,\ i=0,2,\cdots, p.
    \end{equation*}
    Therefore
    \begin{align*}
        \psi^p_F(h)&\equiv\sum_{\tau=1}^pw_{\tau+1}d_{0,\tau}\equiv d_{0,p}\\
        &\equiv \sum_{n=0}^{p-1}(-1)^{p-n}w_0^n\sum\limits_{
        \substack{m_1+\cdots m_{p-n}=p\\ 1\leq m_s\leq p+1\\ m_{p-n}\geq 1}}w_{m_1}\cdots w_{m_{p-n}}\\
        &\equiv (-1)^p\sum\limits_{
        \substack{m_1+\cdots m_{p}=p\\ 1\leq m_s\leq p+1\\ m_{p}\geq 1}}w_{m_1}\cdots w_{m_{p}}.
    \end{align*}
    The only possibility in the last summation is $m_s=1$, hence
        $$\psi^p_F(h)\equiv(-1)^pw_1^p=(-1)^p(-h)^p=h^p\ {\rm mod}\ p$$.
\end{proof}
\begin{example}
    We calculate these formulas for small $p$.
    \par
    When $p=2$, we have
    \begin{equation*}
        \alpha^*=\frac{-2}{h}+\frac{-8}{h^4}+\frac{96}{h^{7}}+O(h^{-10})
    \end{equation*}
    and
    \begin{align*}
        \psi^2_F(h)&=h^2+\alpha^*-h\cdot (\alpha^*)^2\\
        &=h^2-\frac{6}{h}-\frac{40}{h^4}-\frac{544}{h^7}+O(h^{-10}).
    \end{align*}
    When $p=3$, we have
    \begin{equation*}
        \alpha^*=\frac{3}{h}+\frac{108}{h^3}-\frac{162}{h^4}+\frac{7857}{h^5}+O(h^{-6})
    \end{equation*}
    and
    \begin{equation*}
        \psi^3_F(h)=h^3-6h^2-96h+594-\frac{1158}{h}+\frac{14580}{h^2}+\ lower\ terms.
    \end{equation*}
\end{example}
\begin{remark}
    In the $p=3$ case, this power operation formula is different from which in \cite[Section 5.4]{zhu2014power}. This is because the equation for $\alpha$ in \cite{zhu2014power} is not of the form as \ref{e3.2}, but these two equations are equivalent \cite[Remark 2.25]{zhu2019semistable}. In the semi-stable model of Morava $E$-theory \cite[Definition 2.23, Mod.$1^+$]{zhu2019semistable}, it is required that ${\rm Frob}^2=(-1)^{p-1}[p]$, for instance, $[3]$ in this case. While in \cite{zhu2014power}, the model used is ${\rm Frob}^2=[-3]$.
\end{remark}
\begin{remark}
    The formula \ref{e3.5} relies on the $E_\infty$ structure on $F$. In our analysis, we equipped $F$ with the $E_\infty$ structure induced from $E$ via localization. However, $F$ itself may admit a different $E_\infty$ structure. See \cite[Section 6]{vankoughnett2021localizations}.
\end{remark}
\subsection{Connections with elliptic curves}\label{s3.2}

In this section, we state how these computations interact with elliptic curves and $p$-divisible groups.\par
Suppose $C$ is a supersingular elliptic curve over a perfect field $k$ with characteristic $p$. The formal group $\widehat{C}$ associated with $C$ is of height 2. Hence we can transport computations in topology to computations on elliptic curves. This is the initial idea of all explicit computations of height 2 Morava $E$ theories. Rezk calculates the $p=2$ case \cite{rezk2008power} and Zhu calculates the $p=3$ case \cite{zhu2014power}.\par
To be explicit, let $\mathscr{M}_N$ be the moduli stack of elliptic curves equipped with $\Gamma_1(N)$ structure, i.e. an $N$ torsion point. Over $\mathbb{Z}[1/N]$, the moduli problem of $[\Gamma_1(N)]$ is representable, i.e. $\mathscr{M}_N/\mathbb{Z}[1/n]$ is a scheme. Choose a supersingular locus on $\mathscr{M}_N$, we can produce a height 2 formal group as stated above. Since $C$ is supersinguar, the formal group $\widehat{C}$ equals to the $p$-divisible group $C[p^\infty]$ of $C$.
By this, a deformation of $\widehat{C}$ is the same as a deformation of $C[p^\infty]$, which is equivalent to a deformation of $C$ by the Serre-Tate's theorem \cite{tate1967p}. Hence we can construct a universal deformation $C_u$ of $C$ defined over certain ring $R$, with the formal group $\widehat{C_u}$ being the universal deformation of $\widehat{C}$. Then one constructs a corresponding Morava $E$ theory of height 2 associated with $C_u$ over $R$, denoted by $E$, via the Landweber exact functor theorem.\par
To calculate $E^0B\Sigma_p/I$, it suffices to find the place where the universal degree $p$ subgroup $K$ of $C_u$ is defined, for then $K$ is also the universal degree $p$ subgroup of $C_u[p^\infty]=\widehat{C_u}$. This procedure is feasible guaranteed by the moduli problem $\mathscr{M}_p$ is relative representable and hence the simultaneous moduli problem $[\Gamma_1(N)]\times[\Gamma_0(p)]$ is representable by a scheme $\mathscr{M}_{N,p}$ \cite{katz1985arithmetic}. In practice, one usually calculates the coordinates of a point of exact order $p$ to find the explicit expression of $E^0B\Sigma_p/I$ \cite{rezk2008power, zhu2014power}, though these calculations are somehow ad hoc for different primes $p$.
\begin{remark}
    Since in general, an elliptic curve will have $p+1$ subgroups of degree $p$. The moduli scheme $\mathscr{M}_{N,p}$ is of rank $p+1$ over $\mathscr{M}_N$, which coincides with the rank of $E^0B\Sigma_p/I$ over $E^0$.
\end{remark}
\begin{remark}
    Zhu identifies the parameter $\alpha$ which parametrizes subgroups with a modular form of level $[\Gamma_0(p)]$. He then computes the value of $\alpha$ at cusps of $\mathscr{M}_{N,p}$ and uses this to derived the general formula \ref{e3.2} of $E^0B\Sigma_p/I$ for arbitrary primes. \cite{zhu2019semistable}
\end{remark}
Recall that the total power operation $\psi^p_E:E^0\rightarrow E^0B\Sigma_p/I$ stands for taking the target of the universal deformation of Frobenius. It can also be viewed as taking the target curve of the universal degree $p$ isogeny as explained above.
\par
Let $\mathscr{C}_N$ be the universal curve of the moduli problem $[\Gamma_1(N)\times[\Gamma_0(N)]$ over $\mathscr{M}_{N,p}$. There is an isogeny $\Psi^p:\mathscr{C}_N\rightarrow\mathscr{C}_N/\mathscr{G}_N^{(p)}$, with $\mathscr{G}_N^{(p)}$ the universal degree $p$ subgroup of $\mathscr{C}_N$, $P_0$ the $N$ torsion point:
\begin{equation*}
    \left(\mathscr{C}_N,\ P_0,\ du,\ \mathscr{G}_N^{(p)}\right)\mapsto\left(\mathscr{C}_N/\mathscr{G}_N^{(p)},\ \Psi^p(P_0),\ d\Tilde{u},\ \mathscr{C}_N[p]/\mathscr{G}_N^{(p)} \right).
\end{equation*}
Hence it induces an exotic endomorphism of $\mathscr{M}_{N,p}$ \cite[Chapter 11]{katz1985arithmetic}, \cite[Section 2.3]{zhu2019semistable}, so called \textit{the\ Atkin\ Lehner\ involution}. For a supersingular elliptic curve $S$, this Atkin Lehner involution takes $S$ to itself. Therefore it restricts to an endomorphism of the formal neighborhood around the supersingular locus. The previous argument implies that the total power operation is
\begin{equation*}
    \psi^p_E:E^0\hookrightarrow E^0B\Sigma_p/I\xrightarrow{\omega}E^0B\Sigma_p/I,
\end{equation*}
where $\omega$ is the restriction of the Atkin Lehner involution to the formal neighborhood of the given supersingular locus. It is determined by $\psi^p_E(h)=\widetilde{h}$, where $\widetilde{h}$ is the image of $h$ under the Atkin Lehner involution. The calculations along these ideas can be found in \cite[Example 2.14]{zhu2020hecke}.\par
Over $F^0$, the $p$-divisible group $\mathbb{G}_E$ becomes an extension
\begin{equation*}
    0\rightarrow\mathbb{G}_F=\mathbb{G}_E^0\rightarrow\mathbb{G}_E\rightarrow\mathbb{Q}_p/\mathbb{Z}_p\rightarrow 0
\end{equation*}
where $\mathbb{G}_E^0$ is the connected component of $\mathbb{G}_E$ over $F^0$. Or equivalently
\begin{equation*}
    0\rightarrow\widehat{C_u} \rightarrow C_u[p^\infty] \rightarrow \mathbb{Q}_p/\mathbb{Z}_p \rightarrow 0
\end{equation*}
over $F^0$. The map $t:E^0B\Sigma_p/I\rightarrow F^0$ in \ref{e3.1} classifies a degree $p$ cyclic subgroup of $C_u$ over $F^0$. However, in this case, $C_u$ has only one cyclic subgroup of degree $p$, which coincides with the solution of $w(h,\alpha)$ in $F^0$ being unique, or equivalently, the map $t$ being the unique map from $E^0B\Sigma_p/I$ to $F^0$, as stated in Proposition \ref{p3.4}. Moreover, this subgroup is also the unique subgroup of degree $p$ of $\widehat{C_u}=\mathbb{G}_F$ over $F^0$.\par
Therefore, in the interpretation of elliptic curves, we can explain the diagram \ref{e3.1} as follow.
\begin{equation*}
    \begin{tikzcd}
        C_u \ar[r, mapsto, "\psi^p_E"] \ar[d, mapsto] & C_u/K \ar[d, mapsto, "t"] \\
        C_u' \ar[r, mapsto, "\psi^p_F"] & C_u'/H
    \end{tikzcd}
\end{equation*}
where $C_u'$ is the base change of $C_u$ over $F^0$, and $H$ is the degree $p$ cyclic subgroup of $C_u'$ explained above. The maps $\psi^p_E$ and $\psi^p_F$ take the target curves of degree $p$ isogenies starting from $C_u$ over $E^0B\Sigma_p/I$ and $F^0$ respectively. And the map $t$ transform $C_u$ to $C_u'$ and $K$ to $H$, hence it takes the curve $C_u/K$ to $C_u'/H$. The element $\psi^p_F(h)$ can be viewed as the Atkin Lehner involution $\widetilde{h}$ restricted over $F^0$.\par
In the interpretation of formal groups, we have
\begin{equation*}
    \begin{tikzcd}
        \mathbb{G}_E \ar[r, mapsto, "\psi^p_E"] \ar[d, mapsto] & (\psi^p_E)^*\mathbb{G}_E=\mathbb{G}_E/K \ar[d, mapsto, "t"]\\
        \mathbb{G}_F \ar[r, mapsto, "\psi^p_F"] & (\psi^p_F)^*\mathbb{G}_F=\mathbb{G}_F/H
    \end{tikzcd}
\end{equation*}
where $K$ is the universal degree $p$ subgroup of the formal group $\mathbb{G}_E$ and $H$ is the unique degree $p$ subgroup of $\mathbb{G}_F$. The groups $K$ and $H$ are the same thing as which appear in the interpretation of elliptic curves.
\begin{remark}
    Though the map $t$ takes the universal degree $p$ subgroup $K$ of $\mathbb{G}_E$ to the subgroup $H$ of $\mathbb{G}_F$, we can not conclude this from the Strickland's Theorem \cite[Theorem 10.1]{strickland1997finite} directly, due to the discontinuity of $t$.
\end{remark}
\section{Augmented Deformation Spectra}\label{s4}
Let $\mathbb{H}$ be any height $n-1$ formal group over $k(\!(u_{n-1})\!)$. In section \ref{1.3}, we have constructed the universal deformation $\mathbb{H}^u$ over the ring
\begin{equation*}
    F^0=W(k)(\!(u_{n-1})\!)^\wedge_p[\![u_1,\dots,u_{n-2}]\!],
\end{equation*}
and let $$F^*=F^0[\beta^\pm]$$ with $|\beta|=-2$.
The ring $F^*$ is Landweber exact, via the map
\begin{align*}
    MU^*&\rightarrow F^*\\
    x_{2p^i-2}&\mapsto u_iu^{n+1}.
\end{align*}
Hence we can construct a homotopy ring spectrum, called augmented deformation spectrum, denoted by $L_{\mathbb{H}}\in {\rm CAlg}(h{\rm Sp})$, which is complex oriented and carries the formal group $\mathbb{H}^u$.
\subsection{The underlying spectra are equivalent}
In this section we will show that the underlying spectra of $L_{\mathbb{H}}$ are independent of the choice of formal groups $\mathbb{H}$, which means they are all equivalent.\par
Let us recall what happens in Morava $E$-theories. Suppose the field $k$ is perfect, $\mathbb{F}_1$ and $\mathbb{F}_2$ are two formal groups over $k$. Then we have $\mathbb{F}_1$ and $\mathbb{F}_2$ are isomorphic over the algebraic closure $\overline{k}$ of $k$, and in fact, they are isomorphic over the separable closure of $k$.\par
Let $E(k,\mathbb{F}_1)$ and $E(k,\mathbb{F}_2)$ be the corresponding Morava $E$-theories. It has been shown in \cite{luecke2022arentmoravaetheories} that the underlying spectra of them are equivalent, but not as homotopy commutative ring spectra. Hence one can at least take $k$ to be an algebraically closed field, and talk about \textit{the} Morava $E$-theory of height $n$ over it. \par
While things are slightly different when consider formal groups over $k(\!(u)\!)$ even if $k$ is algebraically closed.
\begin{example}
    Let $H$ be the Honda formal group over $\overline{\mathbb{F}}_p(\!(u)\!)$ with its $p$ series given by
    \begin{equation*}
        [p]_H(x)=x^p,
    \end{equation*}
    and let $G$ be the special fiber of the base change of the universal deformation of the height 2 Honda formal group, which is defined over $\overline{\mathbb{F}}_p(\!(u)\!)$ with $p$ series
    \begin{equation*}
        [p]_G=ux^p+_Gx^{p^2}.
    \end{equation*}
    We claim that $H$ and $G$ are not isomorphic over $\overline{\mathbb{F}}_p(\!(u)\!)$.\par
    Suppose $\phi(t)=b_1t+\cdots$ is an isomorphism from $H$ to $G$, hence $b_1\neq 0$. We have
    \begin{align*}
        \phi([p]_H(x))&=[p]_G(\phi(x))\\
        \phi(x^p)&=u{\phi(x)}^p+_G{\phi(x)}^{p^2}.
    \end{align*}
    Calculating with mod $x^{p^2}$, we have
    \begin{align*}
        \sum_{i=1}^{p-1}b_ix^{ip}&=b_1x^p+\cdots+b_{p-1}x^{(p-1)p}\equiv u{\phi(x)}^p\\
        &\equiv ub_1^px^p+ub_2^px^{2p}+\cdots+ub_{p-1}^px^{(p-1)p}\\
        &\equiv \sum_{i=1}^{p-1}ub_i^px^{ip}\ \ {\rm mod}\, x^{p^2}.
    \end{align*}
    This implies that $b_i=ub_i^p$ in $\overline{\mathbb{F}}_p(\!(u)\!)$. But the equation $ux^{p}=x$ does not have a non-zero solution in $\overline{\mathbb{F}}_p(\!(u)\!)$ and $b_1\neq 0$ by the assumption.
\end{example}
\begin{example}
    Following the above notation, we know that the Honda formal group $H$ pocesses all its automorphism over $\mathbb{F}_p$, and in particular, over $\overline{\mathbb{F}}_p(\!(u)\!)$. On the other hand, by the same calculation, one can see that $G$ does not have all its automorphism over $\overline{\mathbb{F}}_p(\!(u)\!)$.
\end{example}
These examples illustrate that one can not talk about \textit{the} augmented deformation spectrum of a given height without specifying the formal group associated to it.\par
To show, despite of their ring structures, their underlying spectra are all equivalent, we need the following lemma.
\begin{lemma}[\cite{luecke2022arentmoravaetheories}, Lemma 7]\label{l4.3}
    Let $R$ be a commutative ring with two Landweber exact formal group laws $e,f:MU_*\rightarrow R$ and let $E$ and $F$ be the corresponding spectra. If there is a ring extension $u:R\rightarrow S$ which is split as an $R$-module map and over $S$, the formal groups $u\circ e$ and $u\circ f$ are isomorphic and Landweber exact, then $E$ and $F$ has the same homotopy type.
\end{lemma}
\begin{theorem}
    Assume $k$ is a perfect field. Let $E$ and $F$ be two formal group laws of height $n-1$ over $k(\!(u_{n-1})\!)$. The corresponding augmented deformation spectra $L_E$ and $L_F$ have the same homotopy types.
\end{theorem}
\begin{proof}
    We may assume $E=H$ is the height $n-1$ Honda formal group law and $F$ is $p$-typical. We have
    \begin{equation*}
        L_H^*=L_F^*=W(k)(\!(u_{n-1})\!)^\wedge_p[\![u_1,\dots,u_{n-2}]\!][\beta^\pm].
    \end{equation*}
    Fix a separable closure of $k(\!(u_{n-1})\!)$, and an isomorphism
    \begin{equation*}
        \Phi(X)=\sum_{i\geq 0}\Phi_1X+_H\Phi_2X^{p^2}+_H+\cdots
    \end{equation*}
    from $F$ to $H$. Let $L$ denote the field $k(\!(u_{n-1})\!)(\Phi_0,\Phi_1,\dots)$, and $$L_n=k(\!(u_{n-1})\!)(\Phi_0,\Phi_1,\dots,\Phi_n).$$
    Therefore $L=\underset{\rightarrow}{\lim}L_i$ and $F$ and $H$ are isomorphic over $L$. \par
    Each $L_i$ is a finite Galois extension of $k(\!(u_{n-1})\!)$ \cite[Section 2.3]{torii2003degeneration}. By \cite[Proposition 4.4]{milne1980etale} or \cite{torii2010comparison}, there is a sequence of finite \'etale $W(k)(\!(u_{n-1})\!)^\wedge_p$ algebras
    \begin{equation*}
        W(k)(\!(u_{n-1})\!)^\wedge_p=B(-1)\rightarrow B(0)\rightarrow\cdots.
    \end{equation*}
    Note that each $B(i)$ is a free $W(k)(\!(u_{n-1})\!)^\wedge_p$ module, so does their limit $B(\infty)$. We have the map $W(k)(\!(u_{n-1})\!)^\wedge_p\rightarrow B(\infty)$ splits as $W(k)(\!(u_{n-1})\!)^\wedge_p$-modules. Since this is a $\mathbb{Z}_p$ module map, this splitting property extends to the $p$-adic completion, $B(\infty)^\wedge_p$, denoted by $B$. By \cite[Lemma 4.2]{torii2010comparison}, the ring $B$ is a Cohen ring of its residue field $L$.\par
    Now Let 
    \begin{align*}
        R=L_H^*=L_F^*&=W(k)(\!(u_{n-1})\!)^\wedge_p[\![u_1,\dots,u_{n-2}]\!],[\beta^\pm]\\
        S&=B[\![u_1,\dots,u_{n-2}]\!][\beta^\pm].
    \end{align*}
    The map $u:R\rightarrow S$ sends $u_i$ to $u_i$ and $\beta$ to $\beta$, which splits as an $R$-module morphism as explained above. Let $\widetilde{F}$ and $\widetilde{H}$ be the universal augmented deformations of $F$ and $H$ respectively, over $R$. It is clear that $u^*\widetilde{F}$ is isomorphic to $u^*\widetilde{H}$ over $S$, because they are deformations of two isomorphic formal group laws. The conclusion follows from the Lemma \ref{l4.3}.
\end{proof}
\bibliography{main}

\end{document}